\documentclass[11pt]{article}
\usepackage[english]{babel}
\usepackage[width=15cm,height=20cm]{geometry}

\usepackage[colorlinks=true,allcolors=blue]{hyperref} 
\newcommand{\doi}[1]{\href{https://doi.org/#1}{doi:#1}}
\newcommand{\myurl}[1]{\href{#1}{#1}}
\usepackage{cite}

\usepackage{tikz}
\usepackage{amsmath,amsfonts,amssymb,mathtools}
\usepackage{enumitem}

\usepackage{amsthm}
\newtheorem{thm}{Theorem}[section]
\newtheorem{prop}[thm]{Proposition}
\newtheorem{lem}[thm]{Lemma}
\newtheorem{cor}[thm]{Corollary}
\theoremstyle{definition}
\newtheorem{definition}[thm]{Definition}
\newtheorem{remark}[thm]{Remark}
\numberwithin{equation}{section}

\newcommand{\pheq}{\phantom{=}\ }
\newcommand{\eqdef}{\coloneqq}
\newcommand{\al}{\alpha}

\newcommand{\Omal}{\Omega_\al}
\newcommand{\omal}{\omega_\al}
\newcommand{\smin}{s_{\al,n}}

\newcommand{\ga}{\gamma}
\newcommand{\be}{\beta}

\newcommand{\om}{\omega}
\newcommand{\eps}{\varepsilon}

\newcommand{\ka}{\varkappa_\al}

\newcommand{\la}{\lambda}
\newcommand{\phihyp}{\varphi_{\al,n}}

\newcommand{\etatrig}{\eta_{\al}}

\newcommand{\lbnd}{\ell_{\al,n}}

\newcommand{\fmin}{f_{\al,n}}

\newcommand{\bC}{\mathbb{C}}
\newcommand{\bN}{\mathbb{N}}
\newcommand{\bR}{\mathbb{R}}

\newcommand{\arccosh}{\operatorname{arccosh}}

\newcommand{\arctanh}{\operatorname{arctanh}}

\newcommand{\clos}{\operatorname{cl}}
\renewcommand{\Re}{\operatorname{Re}}
\renewcommand{\Im}{\operatorname{Im}}
\newcommand{\laasympt}{\lambda^{\operatorname{asympt}}}

\newcommand{\medstrut}{\vphantom{\int_0^1}}
\newcommand{\hstrut}{\mbox{}\ \mbox{}}
\newcommand{\hugestrut}{\vphantom{\displaystyle\int_{0_0}^{1^1}}}
\newcommand{\bigstrut}{\vphantom{\int_{0_0}^{1^1}}}

\title{Eigenvalues
of laplacian matrices \\ of the cycles
with one negative-weighted edge}
\author{Sergei M. Grudsky, Egor A. Maximenko, Alejandro Soto-Gonz\'alez}

\date{\today}

\begin{document}

\maketitle

\begin{center}
    \emph{
    Dedicated to Albrecht B\"ottcher on the occasion of his 70th birthday
    }
\end{center}

\begin{abstract}
We study the individual behavior of the eigenvalues of the laplacian matrices of the cyclic graph of order $n$,
where one edge has weight $\alpha\in\mathbb{C}$, with $\operatorname{Re}(\alpha)<0$,
and all the others have weights $1$.
This paper is a sequel of a previous one where we considered $\operatorname{Re}(\alpha) \in[0,1]$ (Eigenvalues of laplacian matrices of the cycles with one weighted edge, Linear Algebra Appl. 653, 2022, 86--115).
We prove that for $\operatorname{Re}(\alpha)<0$ and $n>\operatorname{Re}(\alpha-1)/\operatorname{Re}(\alpha)$, one eigenvalue is negative while the others belong to $[0,4]$ and are distributed as the function $x\mapsto 4\sin^2(x/2)$.
Additionally, we prove that as $n$ tends to $\infty$, the outlier eigenvalue converges exponentially to $4\operatorname{Re}(\alpha)^2/(2\operatorname{Re}(\alpha)-1)$. 
We give exact formulas for the half of the inner eigenvalues,
while for the others we justify the convergence of Newton's method and fixed-point iteration method.
We find asymptotic expansions, as $n$ tends to $\infty$,
both for the eigenvalues belonging to $[0,4]$ and the outlier.
We also compute the eigenvectors and their norms.


\end{abstract}

\subsection*{Funding}
The research of the first author has been supported by CONAHCYT (Mexico) project ``Ciencia de Frontera'' FORDECYT-PRONACES/61517/2020 and by Regional Mathematical Center of the Southern Federal University with the support of the Ministry of Science and Higher Education of Russia, Agreement 075-02-2021-1386.

The research of the second author has been supported by CONAHCYT (Mexico) project ``Ciencia de Frontera'' FORDECYT-PRONACES/61517/2020 and IPN-SIP projects
(Instituto Polit\'{e}cnico Nacional, Mexico).

The research of the third author has been supported by CONAHCYT (Mexico) PhD scholarship.

\section{Introduction}

For every natural $n\ge 3$ and every $\al$ in $\bC$,
we consider the $n\times n$ complex laplacian matrix $L_{\al,n}$ with the following structure:
\begin{equation*}
L_{\al,8}
=
\begin{bmatrix*}[r]
1+\overline{\al}& -1 & 0 & 0 & 0 & 0 & 0 & -\overline{\al}\phantom{\al} \\
-1\phantom{\al} & 2 & -1 & 0 & 0 & 0 & 0 & 0\phantom{\al} \\
0\phantom{\al} & -1 & 2 & -1 & 0 & 0 & 0 & 0\phantom{\al} \\
0\phantom{\al} & 0 & -1 & 2 & -1 & 0 & 0 & 0\phantom{\al} \\
0\phantom{\al} & 0 & 0 & -1 & 2 & -1 & 0 & 0\phantom{\al} \\
0\phantom{\al} & 0 & 0 & 0 & -1 & 2 & -1 & 0\phantom{\al} \\
0\phantom{\al} & 0 & 0 & 0 & 0 & -1 & 2 & -1\phantom{\al} \\
-\al\phantom{\al} & 0 & 0 & 0 & 0 &  0 & -1 & 1+\al
\end{bmatrix*}.
\end{equation*}
If $\al$ is real, $L_{\al,n}$ is the laplacian matrix of $G_{\al,n}$, where $G_{\al,n}$ is the cyclic graph of order $n$, where the edge between the vertices $1$ and $n$ has weight $\al$, and all other edges have weight $1$. 
See Figure~\ref{fig:graph} for $n=8$.
The eigenvalues and eigenvectors of $L_{\al,n}$ are important to solve the heat and wave equations on $G_{\al,n}$.
See~\cite{M2012} for general theory on laplacian matrices.

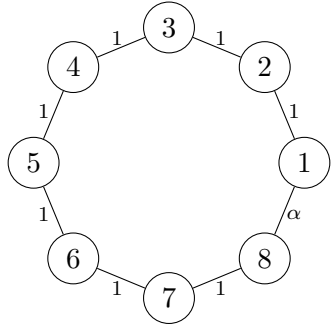
\begin{figure}[htb]
\centering
\begin{tikzpicture}
\foreach \j/\k in {0/1,1/2,2/3,3/4,4/5,5/6,6/7,7/8} {
  \node (N\j) at (\j*360/8:1.8cm) [draw, circle] {$\k$};
}
\foreach \j/\k in {0/1,1/2,2/3,3/4,4/5,5/6,6/7,7/0} {
  \draw (N\j) -- (N\k);
}
\foreach \j in {0,1,2,3,4,5,6} {
  \node at (\j*360/8 + 360/16:1.8cm) {$\scriptstyle 1$};
}
\node at (-360/16:1.8cm) {$\scriptstyle\al$};
\end{tikzpicture}
\caption{Graph $G_{\al,8}$\label{fig:graph}}
\end{figure}

Matrices $L_{\al,n}$ can be considered as tridiagonal Toeplitz matrices with perturbations in the corners $(1,1)$, $(1,n)$, $(n,1)$ and $(n,n)$.
Several investigations in this area and some of its applications have been recently developed, see for example~\cite{BPZ2020,BFGM2014,BYR2006,FK2020,DV2009,GT2009,GMS2021,LWHF2014,OA2014,R2017,TS2017,VHB2018,ZJJ2019}.
These matrices can also be considered as periodic Jacobi matrices.

The present paper is a continuation of~\cite{GMS2022}.
There we proved that for every $\al$ in $\bC$ the characteristic polynomial 
of $L_{\al,n}$,
defined by
$D_{\al,n}(\la)\eqdef\det(\la I - L_{\al,n})$,
equals the characteristic polynomial $D_{\Re(\al),n}$ of $L_{\Re(\al),n}$.
This implies that the eigenvalues of $L_{\al,n}$ only depend on $\Re(\al)$. 
For this reason, we are going to consider $\al$ as a real number.
For $\al$ in $\bR$, these matrices are real and symmetric, their eigenvalues are real, we enumerate them as follows:
\[ 
\la_{\al,n,1} \le \la_{\al,n,2} \le \ldots \le \la_{\al,n,n} . 
\]
It is a very well-known fact that the eigenvalues of the $n\times n$ tridiagonal Toeplitz matrix, with values $-1, 2, -1$ in the non-zero diagonals, are asymptotically distributed as the values of
\begin{equation}\label{eq:g}
    g(x)\eqdef4\sin^2(x/2)\qquad(x\in[0,\pi]).
\end{equation}
on $[0,\pi]$ as $n\to\infty$.
By the Cauchy interlacing theorem~\cite[Theorem 4.2]{SS1990}, it follows that the eigenvalues of $L_{\al,n}$ are also asymptotically distributed by $g$ on $[0,\pi]$, as $n$ tends to infinity.
This is also a very simple application of the theory of generalized locally Toeplitz sequences~\cite{GS2017}.

In~\cite{GMS2022}, we studied the individual behavior of the eigenvalues of the matrices $L_{\al,n}$ for $\al$ in $(0,1)$.
In that case, the eigenvalues of $L_{\al,n}$ belonged to $[0,4]$.
We solved the characteristic equation by numerical methods and derived asymptotic formulas for all eigenvalues.

Now we consider $\al<0$.
This means that the interaction between the vertices $1$ and $n$ has a negative coefficient, while the interactions between vertices $1$ and $2$, $2$ and $3$, $\ldots$\ , and $n-1$ and $n$, have the same positive coefficient.
We do not have physical examples of this situation.

If $n>(\al-1)/\al$, then only one eigenvalue of $L_{\al,n}$ is negative while the others belong to the interval $[0,4]$ and behave as in the case $0<\al<1$, considered in~\cite{GMS2022}.

Commonly, the spectral analysis ignores the eigenvalues outside the clusters.
In this paper, we focus our attention on these eigenvalues, so we introduce the next definitions.
The phrase ``number of eigenvalues'' assumes counting the eigenvalues with their algebraic multiplicities.

\begin{definition}\label{def:outlier_adherent}
Let $N\in\bN$ and $(A_n)_{n\ge N}$ be a matrix sequence where $A_n$ is a $n\times n$ matrix for every $n$.
Suppose that $\Omega\in \bC$.
We say that $\Omega$ is an \emph{outlier adherent point} for $(A_n)_{n\ge N}$ 
if for every sufficiently small $\eps>0$ the number of the eigenvalues of $A_n$ belonging to the $\eps$-neighborhood of $\Omega$ is strictly positive and behaves as $o(n)$ as $n\to\infty$.
\end{definition}

We do not define the concept of outlier eigenvalue for an individual matrix; we speak about outlier eigenvalues for a matrix $A_m$ in the context of a matrix sequence $(A_n)_{n\ge N}$.

\begin{definition}\label{def:outlier_eigenvalue}
Let $N\in\bN$ and $(A_n)_{n\ge N}$ be a matrix sequence where $A_n$ is a $n\times n$ matrix for every $n$.
Assume that $m\ge N$ and $\la$ be an eigenvalue of $A_m$.
We say that $\la$ is an \emph{outlier eigenvalue} for $A_m$ with respect to $(A_n)_{n\ge N}$ if there exists $\eps>0$ such that the number of the eigenvalues of $A_n$ belonging to the $\eps$-neighborhood of $\la$ behaves as $o(n)$ as $n\to\infty$.
\end{definition}

Figuratively speaking,
Definition~\ref{def:outlier_eigenvalue} means that $\la$ is an eigenvalue of $A_m$ and there are not many eigenvalues of $A_n$ near $\la$, when $n$ is sufficiently large.

Obviously, if $\Omega$ is an outlier adherent point for $(A_n)_{n\ge N}$, then there exist $\eps>0$ and $M\in\bN$ such that for every $m\ge M$ the eigenvalues of $A_m$ belonging to the $\eps$-neighborhood are outlier eigenvalues.

The principal novelty of this paper is a thorough analysis of the asymptotic behavior of the outlier eigenvalues for the specific matrix family.
In particular, we prove that the sequence of outlier eigenvalues converges exponentially to the outlier cluster point $\Omal \eqdef 4\al^2/(2\al-1)$, as $n\to\infty$.
The outlier eigenvalues naturally appear in the study of some structured matrices (see~\cite[Example~10.7]{GS2017},~\cite{BS1999}), but we have not found examples of their detailed analysis in the literature, except for~\cite{GMS2021}, where the asymptotic formula is less precise.

The main results of this paper are stated in Sections~\ref{sec:main_left}.
The corresponding proofs are in Sections~\ref{sec:left_char_pol}--\ref{sec:eigvec_left}.
We represent the characteristic polynomial in the convenient form and show the localization of the eigenvalues (Section~\ref{sec:left_char_pol}), the asymptotic behavior of the inner eigenvalues and their computation with the Newton method (Section~\ref{sec:inner_eigvals_left}),
the asymptotic behavior of the outlier eigenvalue (Sections~\ref{sec:out_eig_left}),
and calculate the norms of the eigenvectors (Section~\ref{sec:eigvec_left}).
In Section~\ref{sec:num_exp} we show some numerical experiments.

In comparison to works of other authors, we deal with a rather special matrix family, but this matrix family is not trivial (the eigenvalues and eigenvectors are not given by simple direct formulas), and our results on the eigenvalues and eigenvectors for this family are very complete.

\begin{remark}
    The case $\al>1$ (when $G_{\al,n}$ has one ``overweighted'' edge) is slightly more complicated, and we are going to study it in another paper.
    If $\al>1$ and $n>\al/(\al-1)$, then one eigenvalue of $L_{\al,n}$ is greater than $4$ and the others are in $[0,4]$.
    In that case, as $n$ tends to infinity, the maximal eigenvalue $\la_{\al,n,n}$ converges exponentially to $\Omal = 4\al^2/(2\al-1) > 4$.
    For $\al>1$, the situation essentially depends on the parity of $n$: if $n$ is even, then $\la_{\al,n,n} < \Omal$, and if $n$ is odd, then $\la_{\al,n,n} > \Omal$.
\end{remark}

\section{Main results}
\label{sec:main_left}

Define
\begin{equation}\label{eq:ka}
\ka\eqdef\frac{\al-1}{\al},
\qquad\text{i.e.,}\qquad
\ka=\frac{|\al|+1}{|\al|},
\end{equation}
\begin{equation}
\label{eq:Om_al}
\Omal \eqdef
\frac{4\al^2}{2\al-1},
\qquad\text{i.e.},\qquad
\Omal=-\frac{4}{\ka^2-1}.
\end{equation}
Since $\al<0$, $\ka>1$ and $\Omal<0$.
For every $j$ in $\{1,\ldots,n\}$, we put
\begin{equation*}
d_{n,j}
\eqdef
\frac{(j-1)\pi}{n}.
\end{equation*}

\begin{thm}[eigenvalues' localization]
\label{thm:localization_eigvals_left}
Let $n\ge 3$.
\begin{enumerate}
 \item[1)] If $n<\ka$,
 then $\la_{\al,n,1}= 0$ and $0<\la_{\al,n,2} < g(d_{n,2})$.
 \item[2)] If $n=\ka$,
 then $\la_{\al,n,1} = \la_{\al,n,2} = 0$.
 \item[3)] If $n>\ka$, then $\Omal<\la_{\al,n,1}<0$ and $\la_{\al,n,2} = 0$.
\end{enumerate}
Furthermore, for every $j$ with $3\le j\le n$,
\begin{align*}
    \la_{\al,n,j} &= g(d_{n,j})\qquad  (j\ \text{is odd}), \\
    g(d_{n,j-1})<\la_{\al,n,j} &<g(d_{n,j})\qquad (j \ \text{is even}). 
    \end{align*}
\end{thm}

Theorem~\ref{thm:localization_eigvals_left} implies that $\la_{\al,n,j}$ with odd $j$ does not depend on $\al$.

Theorem~\ref{thm:localization_eigvals_left} provides another proof of the fact that the eigenvalues of $L_{\al,n}$ are asymptotically distributed as the function $g$ on $[0,\pi]$. 

Motivated by Theorem~\ref{thm:localization_eigvals_left} we use $g$ defined by~\eqref{eq:g} as a change of variable in the characteristic equation when  $\la_{\al,n,j}\in[0,4]$ and set
\begin{equation*}
z_{\al,n,j} \eqdef \widetilde{g}^{-1}(\la_{\al,n,j}),
\end{equation*}
where $\widetilde{g}\colon[0,\pi]\to[0,4]$ is a restriction of $g$.

Define $g_-\colon [0,\infty)\to (-\infty,0]$ by \begin{equation}\label{eq:gmin}
    g_-(x)\eqdef2-2\cosh(x) = -4\sinh^2\frac{x}{2}.
\end{equation}
Define 
\begin{equation}\label{eq:N_al}
N_\al \eqdef \max\{3,\lfloor\ka\rfloor+1\}.
\end{equation}
If $n\ge N_\al$, we use~\eqref{eq:gmin} as a change of variable and put
\[
\smin\eqdef g_-^{-1}(\la_{\al,n,1}).
\]
In Figure~\ref{fig:ggmineigvals} we have glued together $g$ and reflected $g_-$ into one spline.

\begin{figure}[htb]
\centering
\includegraphics{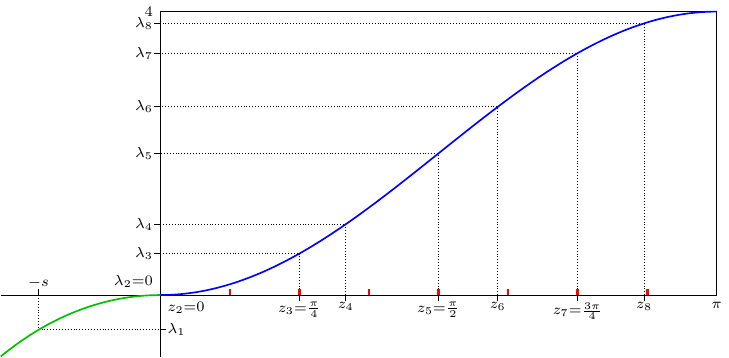}
\caption{Plot of $g$ (blue), plot of $x\mapsto g_-(-x)$ (green), points $z_{\al,n,j}$ and $\smin$, and the corresponding values of $\la_{\al,n,j}$, for $\al = -1/2$ and $n= 8$. 
The red labels on the horizontal axis are $j\pi/n$.
\label{fig:ggmineigvals}
}
\end{figure}

After applying the changes of variables $g$ or $g_-$, the characteristic equation transforms to certain equations for $z_{\al,n,j}$ or $s_{\al,n}$, respectively.
Those equations, stated in  the forthcoming Theorem~\ref{thm:main_equations}, will be written in terms of the following functions $\etatrig$ and $\phihyp$.
We define $\etatrig\colon[0,\pi]\to\bR$ by
\begin{equation}\label{eq:eta}
\etatrig(x)\eqdef2\arctan\left(\ka\tan\frac{x}{2}\right)-\pi = 
 -2\arctan\left(\ka^{-1}\cot\frac{x}{2}\right).
\end{equation}
This function strictly increases and takes values in $[-\pi,0]$.
Define $\phihyp\colon [0,\infty)\to\bR$ by
\begin{equation}\label{eq:phi}
    \phihyp(x) \eqdef 2\arctanh\left(\ka^{-1}\tanh\frac{nx}{2}\right).
\end{equation}
This function strictly increases on $[0,\infty)$ and takes values from $0$ to 
\begin{equation}\label{eq:om_al}
 \omal \eqdef \log(1-2\al).
\end{equation}
Notice that $g_-(\omal) = \Omal$ and $\tanh(\omal/2) = \ka^{-1}$. 

\begin{thm}[main equations]
\label{thm:main_equations}
Let $n\ge N_\al$.
Then $\smin$ is the unique solution on $(0,\omal)$ of the equation
\begin{equation}\label{eq:eq_first}
x = \phihyp(x).
\end{equation}
For every even $j$ with $4\le j\le n$, the number 
$z_{\al,n,j}$ is the unique solution on $[0,\pi]$ of the equation
\begin{equation}\label{eq:ls_weak_eq}
x
=d_{n,j} + \frac{\etatrig(x)}{n}.
\end{equation}
\end{thm}

For $n\ge N_\al$, the main equations from Theorem~\ref{thm:main_equations} can be solved by the fixed point method and Newton's method, see details in Theorems~\ref{thm:contraction_inner_left},~\ref{thm:Newton},~\ref{thm:phi_fixed_point},~\ref{thm:linear_conv_first_eigval}.
For $n< N_\al$, we only guarantee the convergence of the bisection method, see Remark~\ref{rem:n<ka}.

Using equations from Theorem~\ref{thm:main_equations} we derive asymptotic expansions for $\la_{\al,n,j}$ as $n\to\infty$.
The expansion of $\la_{\al,n,1}$ is drastically different from the expansion of the inner eigenvalues.
Therefore, we state the corresponding results in two separate theorems.

We define $\Lambda_{\al,n}\colon [0,\pi]\to\bR$ by
\begin{equation*}
\Lambda_{\al,n}(x)
\eqdef g(x)
+\frac{g'(x)
\etatrig(x)}{n}
+\frac{g'(x)
\etatrig(x)\etatrig'(x)
+\frac{1}{2}g''(x)
\etatrig(x)^2}{n^2}.
\end{equation*}
For all even $j$ with $2\le j\le n$,
we define $\laasympt_{\al,n,j}$ by
\begin{equation}\label{eq:laasympt_w}
\laasympt_{\al,n,j}
\eqdef \Lambda_{\al,n}(d_{n,j}).
\end{equation}

\begin{thm}[asymptotic expansion of inner eigenvalues]\label{thm:inner_eigvals_expansion}
There exists $C_1(\al)>0$ such that for every $n\ge N_\al$,
\begin{equation}\label{eq:weak_lambda_asympt_K}
\max_{\substack{4\le j\le n \\
j\, \text{even}}} \left|\la_{\al,n,j}-\laasympt_{\al,n,j}\right|
\le\frac{C_1(\al)}{n^3}.
\end{equation}
\end{thm}

To state the asymptotic formula for $\la_{\al,n,1}$, we introduce the following numbers:
\begin{equation}\label{eq:bs}
    \be_{\al,1}
\eqdef
    \frac{16\al^2(\al-1)^2}{(1-2\al)^2}
,
\quad
\be_{\al,2}
\eqdef \frac{64\al^3 (\al-1)^3}{(1-2\al)^3}
,
\quad
\be_{\al,3}
\eqdef \frac{32\al^2 (1-\al)^2 (2\al^2-2\al+1)}{(1-2\al)^3}
.
\end{equation}
Equivalently,
\begin{equation}\label{eq:bs2}
\be_{\al,1}
=
\frac{16\ka^2}{(\ka^2-1)^2},
\qquad
\be_{\al,2}
=
\frac{64\ka^3}{(\ka^2-1)^3},
\qquad
\be_{\al,3}
=
\frac{32\ka^2(\ka^2+1)}{(\ka^2-1)^3}.
\end{equation}
We define $\laasympt_{\al,n,1}$ by
\begin{equation}\label{eq:left_laasympt}
  \laasympt_{\al,n,1} \eqdef \Omal +  \be_{\al,1} e^{-n\omal} +  \be_{\al,2}ne^{-2n\omal}
 -   \be_{\al,3} e^{-2n\omal}.
\end{equation}

\begin{thm}[asymptotic expansion of the first eigenvalue]
\label{thm:left_strong_asympt}
As $n\to\infty$,
the extreme eigenvalue $\la_{\al,n,1}$
of $L_{\al,n}$ converges exponentially
to $\Omal$.
More precisely, there exists $C_2(\al)>0$ such that for every $n\ge N_\al$,
\begin{equation}\label{eq:asmpt_min_three_terms}
\begin{aligned}
 \left|\la_{\al,n,1} - \laasympt_{\al,n,1} \right| \le C_2(\al)n^2 e^{-3n\omal}.
\end{aligned}
\end{equation}
\end{thm}

Since $e^{-\omal} = 1/(1+2|\al|)$, the expression $e^{-n\omal}$ can be written as $1/(1+2|\al|)^n$.

So, if $\al<0$ and $n$ is large enough, the minimal eigenvalue goes out of $[0,4]$ and converges rapidly to the negative number $\Omal$, and the situation with the other eigenvalues is similar to the case $0<\al<1$, but there is no eigenvalue in the interval $(0,g(2\pi/n))$.
The ``left spectral gap'' equals $\la_{\al,n,2}-\la_{\al,n,1} = |\la_{\al,n,1}| $ and converges exponentially to $|\Omal|$ as $n$ goes to infinity.

In particular, we conclude that Definitions~\ref{def:outlier_adherent} and~\ref{def:outlier_eigenvalue} make sense for our matrix sequence: $\Omal$ is an outlier adherent point for $(L_{\al,n})_{n\ge3}$, and $\la_{\al,n,1}$ for $n\ge N_\al$ is an outlier eigenvalue for $L_{\al,n}$.

Finally, we focus our attention on the eigenvectors.
In general, for complex values of $\al$, the eigenvectors depend on $\al$, not only on $\Re(\al)$.

\begin{thm}[eigenvectors for $\Re(\al)<0$]
\label{thm:eigvec_ls}
Let $\al\in\bC$ with $\Re(\al)<0$
and $n\ge N_{\Re(\al)}$.
Then $L_{\al,n}$ has the following eigenvectors.
\begin{enumerate}
\item
$[1,\ldots,1]^\top$ is an eigenvector associated to the eigenvalue $\la_{\al,n,2} =0$.
\item
For every $j$, $3\le j \le n$, the vector $v_{\al,n,j}=[v_{\al,n,j,k}]_{k=1}^n$ with the following components is an eigenvector associated to $\la_{\al,n,j}$:
\begin{equation}\label{eq:eivec_w}
    v_{\al,n,j,k}\eqdef \sin(k z_{\al,n,j}) -(1-\overline{\al}) \sin((k-1)z_{\al,n,j})  + \overline{\al} \sin((n-k)z_{\al,n,j}).
\end{equation}
\item
The vector
$v_{\al,n,1}=[v_{\al,n,1,k}]_{k=1}^n$ with the following components is an eigenvector associated to $\la_{\al,n,1}$:
\begin{equation}\label{eq:eivec_ls}
    v_{\al,n,1,k} \eqdef \sinh(k\smin)  - (1-\overline{\al})\sinh((k-1)\smin) + \overline{\al}\sinh((n-k)\smin).
\end{equation}
\end{enumerate}
\end{thm}

If $n< \varkappa_{\Re(\al)}$, then $[1,\ldots,1]^\top$ is an eigenvector associated to $\la_{\al,n,1} = 0$, and 
the components of an eigenvector associated to $\la_{\al,n,2}$ can be computed by~\eqref{eq:eivec_w}.

If $\al < 0$ and $n = \ka$ (i.e., $\al = -1/(n-1)$), then the eigenvalue $0$ has two orthogonal eigenvectors: $[1,\ldots,1]^\top$ and $[-(n-1),-(n-3), \ldots,n-3,n-1]^\top$.

To approximate the norms of the eigenvectors, we define
\begin{equation}\label{eq:nu_al}
  \nu_{\al}(x) \eqdef \frac{1-\Re(\al)}{2}g(x) - \frac{\Re(\al)}{2} g(\eta_{\Re(\al)}(x)) + \frac{\Re(\al)-|\al|^2}{2}g(x-\eta_{\Re(\al)}(x)) + 2|\al|^2,
\end{equation}
\begin{equation}\label{eq:mu}
\mu_\al \eqdef \frac{|\al|}{2\sqrt{2(\Re(\al)^2-\Re(\al))}}.
\end{equation}

\begin{thm}[norms of eigenvectors for $\Re(\al)<0$]
\label{thm:norms_eigvec_ls}
Let $\al\in\bC$ with $\Re(\al)<0$
and $n\ge N_{\Re(\al)}$.
\begin{enumerate}
\item
If $j\ge 3$ is odd, then
\begin{equation}\label{eq:norm_eigvec_odd_left}
    \|v_{\al,n,j}\|_2 = |1-\al|\sqrt{\frac{n}{2}\la_{\al,n,j}}.
\end{equation}
\item
If $j\ge4$ is even, then
\begin{equation}\label{eq:norm_eigvec_even_left}
    \|v_{\al,n,j}\|_2 = \sqrt{\nu_{\al}(d_{n,j}) n} + O_\al\left(\frac{1}{\sqrt{n}}\right),
\end{equation}
with $O_\al\left(\frac{1}{\sqrt{n}}\right)$ uniformly on $j$.
\item As $n\to\infty$,
\begin{equation}\label{eq:norm_eigvec_left_1}
\|v_{\al,n,1}\|_2
=\mu_{\al}\,e^{n\omal} + O(n) = \mu_{\al} (1-2\Re(\al))^n + O(n).
\end{equation}
\end{enumerate}
\end{thm}

In Propositions~\ref{prop:exact_norm_eigenvector} and~\ref{prop:norm_v1_expansion} we state exact formulas for the eigenvectors, but they are more complicated.

\section{The characteristic polynomial and eigenvalues' localization}\label{sec:left_char_pol}

In this section, we repeat some formulas for $D_{\al,n}(\la)\eqdef\det(\la I - L_{\al,n})$ established in~\cite{GMS2022} and prove Theorem~\ref{thm:localization_eigvals_left}.
Propositions~\ref{prop:char_pol},~\ref{prop:pol_char_factorized} and Lemmas~\ref{lem:q_n_eval_jpi/n},~\ref{lem:limits_pq} were proved in~\cite[Section 4]{GMS2022}. Proposition~\ref{prop:char_pol} also follows from~\cite[Corollary 2.4]{FF2009}.

For every $m$ in $\{0\} \cup \bN$, we denote by $T_m$ and $U_m$ the $m$th degree Chebyshev polynomials of the first and second kind, respectively. 
They are determined by the following properties:
\begin{equation}\label{eq:Chebyshev_poly}
T_{m}\left(\frac{t+t^{-1}}{2}\right) = \frac{t^m+t^{-m}}{2},\qquad  U_m\left(\frac{t+t^{-1}}{2}\right) = \frac{t^{m+1}-t^{-m-1}}{t-t^{-1}}.
\end{equation}

\begin{prop}[characteristic polynomial of $L_{\al,n}$ for complex $\al$]\label{prop:char_pol}
For $n\ge 3$ and $\al\in\bC$,
\begin{equation*}
D_{\al,n}(\la)
= (\la-2\Re(\al)) U_{n-1}\left(\frac{\la-2}{2}\right)
-2\Re(\al) U_{n-2}\left(\frac{\la-2}{2}\right)+2(-1)^{n+1}\Re(\al).
\end{equation*}
\end{prop}

In the rest of the section, we suppose that $\al<0$.

\begin{prop}\label{prop:pol_char_factorized}
For $n\ge 3$,
\begin{equation}\label{eq:char_pol_fact}
D_{\al,n}(4-t^2) = 2(-1)^n \frac{p_{n}(t) q_{\al,n}(t)}{t},
\end{equation}
where
\begin{equation}\label{eq:pq_fact}
p_{n}(t) = (t^2-4)
U_{n-1}\left(\frac{t}{2}\right), \qquad q_{\al,n}(t) =  (1-\al)T_n\left(\frac{t}{2}\right)
+\al \frac{t}{2} U_{n-1}\left(\frac{t}{2}\right).
\end{equation}
\end{prop}

The polynomials~\eqref{eq:pq_fact} after the change of variable $t = 2\cos(x/2)$ read as
\[ 
p_n(2\cos(x/2)) = -4\sin\frac{x}{2} \sin\frac{nx}{2},  \qquad q_{\al,n}(2\cos(x/2)) = (1-\al) \cos\frac{nx}{2} + \al\cos\frac{x}{2} \frac{\sin\frac{nx}{2}}{\sin\frac{x}{2}}. 
\]
Taking into account that $g(x)=4-(2\cos(x/2))^2$,
\begin{equation}
\label{eq:charpol_factorization_trig}
D_{\al,n}(g(x)) =  (-1)^{n+1} \frac{4\sin\frac{x}{2}\sin\frac{nx}{2}}{\cos\frac{x}{2}} \left( (1-\al)\cos\frac{nx}{2} + \al \cos\frac{x}{2} \frac{\sin\frac{nx}{2}}{\sin\frac{x}{2}}\right).
\end{equation}
Analogously, after the change of variable $t=2\cosh(x/2)$,~\eqref{eq:char_pol_fact} transforms to
\begin{equation}\label{eq:char_pol_gminus}
    D_{\al,n}(g_-(x)) = (-1)^n \frac{4\sinh\frac{x}{2} \sinh\frac{n x}{2}}{\cosh\frac{x}{2}} \left( (1-\al) \cosh\frac{n x}{2} + \al \cosh\frac{x}{2} \frac{\sinh\frac{n x}{2}}{\sinh\frac{x}{2}} \right).
\end{equation}

\begin{lem}\label{lem:q_n_eval_jpi/n}
For every $j$ with $1\le j\le n-1$,
\begin{equation*}
    q_{\al,n}\left(2\cos\frac{j\pi}{2n}\right) = \begin{cases}
         (-1)^{\frac{j}{2}} (1-\al), & \quad \text{if } j \text{ is even}, \\
        (-1)^{\frac{j-1}{2}} \al\cot\frac{j\pi}{2n} , & \quad \text{if } j \text{ is odd}.
    \end{cases}
\end{equation*}
Moreover,
\begin{equation*}
    q_{\al,n}(0) = 
    \begin{cases}
        0, & \quad \text{if}\ n\ \text{is odd}, \\
        (-1)^{\frac{n}{2}} (1-\al), &\quad \text{if}\ n\ \text{is even},
    \end{cases}
     \qquad q_{\al,n}(2) = (1-\al) + \al n.
\end{equation*}
\end{lem}

\begin{lem}\label{lem:limits_pq}
If $n$ is odd, then
\begin{equation*}
    \lim_{t\to 0^+} \frac{2q_{\al,n}(t)}{t} = (-1)^{\frac{n-1}{2}} \Bigl( \al + (1-\al)n \Bigr) ,
\end{equation*}
and if $n$ is even, then
\begin{equation*}
    \lim_{t\to 0^+} \frac{2p_{n}(t)}{t} = 4(-1)^{\frac{n}{2}} n.
\end{equation*}
\end{lem}

\begin{prop}[trivial eigenvalues of $L_{\al,n}$]\label{prop:trivial_eigvals}
For every $n\ge3$ and every even $j$ with $0\le j\le n-1$, the number $g(j\pi/n)$ is an eigenvalue of $L_{\al,n}$.
\end{prop}

\begin{proof}
    Follows from~\eqref{eq:pq_fact} or~\eqref{eq:charpol_factorization_trig}.
\end{proof}

For every $j$ with $1\le j\le n$, we define \[
I_{n,j} \eqdef \left(\frac{(j-2)\pi}{n}, \frac{(j-1)\pi}{n}\right) = (d_{n,j-1},d_{n,j}).
\] 

\begin{proof}[Proof of Theorem~\ref{thm:localization_eigvals_left}]
From Lemmas~\ref{lem:q_n_eval_jpi/n} and~\ref{lem:limits_pq} we obtain the following facts.
\begin{enumerate}[noitemsep]
    \item If $n<(\al-1)/\al$, then $q_{\al,n}(2\cos(x/2))$ changes its sign in the interval $I_{n,2}$.
    \item If $n=(\al-1)/\al$, then $q_{\al,n}(2) = 0$.
    \item If $n > (\al-1)/\al$, then $q_{\al,n}(t)$ changes its sign in the interval $(2,r_\al+r_\al^{-1})$ where  $r_\al \eqdef
    \sqrt{1-2\al}$.
    Indeed,     by~\eqref{eq:Chebyshev_poly} and the equalities $1-\al = (1+r_\al^2)/2$, $\al = (1-r_\al^2)/2$ it follows that
    \[
    q_{\al,n}\left(r_\al + \frac{1}{r_\al}\right) = \frac{1}{2} (1+r_\al^2)  r_\al^{-n} > 0.
    \]
    Lemma~\ref{lem:q_n_eval_jpi/n} and assumption $n> \ka$ imply that $q_{\al,n}(2) <0$.
    Finally, a simple computation shows that
    \[
    \Omal = 4 - (r_\al + r_\al^{-1})^2.
    \]
    \end{enumerate}
Hence, we obtain the statements of the theorem about $\la_{\al,n,1}$ and $\la_{\al,n,2}$.
For the rest of the eigenvalues, the proof is similar to the proof of~\cite[Theorem 1]{GMS2022}.
In particular, for odd $j$, we use Proposition~\ref{prop:trivial_eigvals}.
\end{proof}


\section{Inner eigenvalues}
\label{sec:inner_eigvals_left}

In this section, we suppose that $\al<0$.
If $\la\in(0,4)$, we use the change of variable $\la=g(x)$, with $x\in(0,\pi)$, in~\eqref{eq:char_pol_fact}.
Then, $D_{\al,n}(g(x)) = 0$ reduces to $q_{\al,n}(2\cos(x/2)) = 0$, which is equivalent to  
\begin{equation*}
    \tan\frac{nx}{2} = \ka\tan\frac{x}{2}.
\end{equation*}
In particular, 
for even $j$ with $4\le j\le n$,
the solution $z_{\al,n,j}$ belonging to $I_{n,j}$ satisfies~\eqref{eq:ls_weak_eq}.
Thereby we obtain the second part of Theorem~\ref{thm:main_equations}.

Equation~\eqref{eq:ls_weak_eq} from Theorem~\ref{thm:main_equations}
can be rewritten in the form
\begin{equation}
\label{eq:equation_with_eta}
nx-(j-1)\pi=\etatrig(x).
\end{equation}
Figure~\ref{fig:eta_equation} shows $\eta_\al$ and the left-hand side of~\eqref{eq:equation_with_eta} for a couple of examples.

\begin{figure}[htb]
\centering
\includegraphics{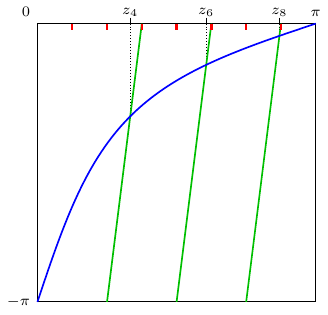}
\qquad
\includegraphics{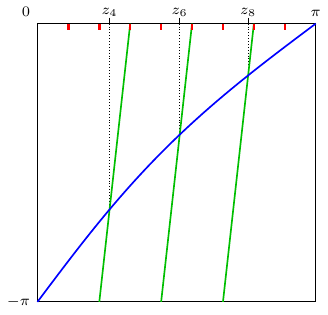}
\caption{Plot of $\eta_\al$ (blue) and the left-hand side of \eqref{eq:equation_with_eta}
(green)
for $\al = -1/2$, $n=8$ (left)
and $\al=-3$, $n=9$ (right).
\label{fig:eta_equation}
}
\end{figure}

Proposition~\ref{prop:eta_bound} and Theorem~\ref{thm:contraction_inner_left} follow directly from the properties of $\etatrig$, similarly to~\cite[Propositions 21 and 22]{GMS2022}.
The first two derivatives of $\etatrig$ are
\begin{align}
\label{eq:eta_d1}
    \etatrig'(x) & = \frac{\ka \left(1+\tan^2\frac{x}{2}\right)}{1+\ka^2\tan^2\frac{x}{2}} = \frac{1}{\ka} + \frac{\ka^2-1}{\ka\left(1+\ka^2\tan^2\frac{x}{2}\right)}, 
    \\
    \label{eq:eta_d2}
    \etatrig''(x) & = - \frac{\ka\left(\ka^2-1\right) \left(1+\tan^2\frac{x}{2}\right) \tan\frac{x}{2} }{\left(1+\ka^2\tan^2\frac{x}{2}\right)^2}.
\end{align}  

\begin{prop}\label{prop:eta_bound}
Each derivative of $\etatrig$ is a bounded function on $(0,\pi)$. In particular, 
\begin{align*}
    \sup_{0<x<\pi}|\etatrig'(x)| &= \ka, &
    \sup_{0<x<\pi}|\etatrig''(x)| & \le \frac{\ka^2-1}{2}.
\end{align*}
\end{prop}

\begin{thm}\label{thm:contraction_inner_left}
    Let $n\ge N_\al$, $j$ be even, $4\le j\le n$. 
    Then the function $x\mapsto d_j + \etatrig(x)/n$ is a contraction on $\clos(I_{n,j})$, and its fixed point is $z_{\al,n,j}$.
\end{thm}

Due to Theorem~\ref{thm:contraction_inner_left}, $z_{\al,n,j}$ can be computed by the simple iteration method.
Figure~\ref{fig:fixed_point_inner_eigenvalues} shows the functions from Theorem~\ref{thm:contraction_inner_left} for a couple of examples.

\begin{figure}[htb]
\centering
\includegraphics{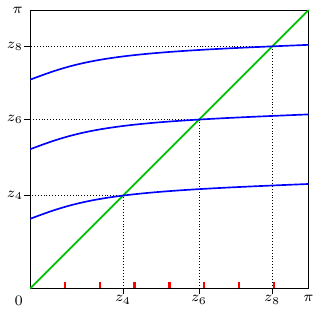}
\qquad
\includegraphics{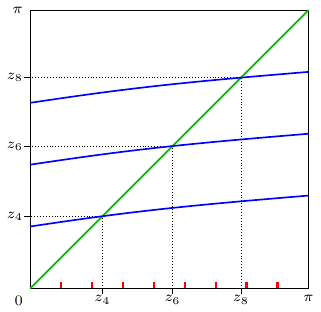}
\caption{Plots of the functions from Theorem~\ref{thm:contraction_inner_left} and their fixed points, for $\al = -1/2$, $n=8$ (left)
and $\al=-3$, $n=9$ (right).
\label{fig:fixed_point_inner_eigenvalues}
}
\end{figure}

In the upcoming Proposition~\ref{prop:convergence_f_convex_y0>c}
we recall some sufficient conditions for the  convergence of Newton's method for convex functions and provide an upper bound for the linear convergence.
The corresponding proofs appear in~\cite{GMS2022}, or in \cite[Section 22, Problem 14]{Spivak1994} and \cite[Theorem 2.2]{A1989}.

\begin{prop}
\label{prop:convergence_f_convex_y0>c}
Let $a,b\in\bR$ with $a<b$,
$F$ is differentiable and $F'>0$ on $[a,b]$,
$F$ is convex on $[a,b]$,
$c\in[a,b]$,
$F(c)=0$,
$y^{(0)}\in[c,b]$.
Define the sequence $(y^{(m)})_{m=0}^\infty$
by the recurrence relation
\begin{equation*}
    y^{(m+1)} = y^{(m)} - \frac{F\left(y^{(m)}\right)}{F'\left(y^{(m)}\right)}.
\end{equation*}
Then $y^{(m)}$ belongs to $[c,b]$ for every $m\ge0$, the sequence $(y^{(m)})_{m=0}^\infty$ decreases and converges to $c$, with
\begin{equation*}
y^{(m)}-c
\le (b-a)\left(1 - \frac{F'(a)}{F'(b)}\right)^m.
\end{equation*}
\end{prop}

For every $n\ge4$ and every even $j$ with $4\le j\le n$, we define $h_{\al,n,j}\colon \clos(I_{n,j})\to \bR$ by
\begin{equation*}
    h_{\al,n,j}(x) \eqdef n x-(j-1)\pi-\etatrig(x).
\end{equation*}
Recall that $N_\al$ is defined by~\eqref{eq:N_al}.

\begin{thm}[convergence of Newton's method]
\label{thm:Newton}
Let $n \ge N_\al$, $j$ be even, $4\le j\le n$
and $y_{\al,n,j}^{(0)}=d_{n,j}$.
Define the sequence $(y_{\al,n,j}^{(m)})_{m=0}^\infty$ by the recursive formula
\begin{equation}\label{eq:Newton_sequence_eigval} 
y_{\al,n,j}^{(m)}
\eqdef y_{\al,n,j}^{(m-1)} - \frac{h_{\al,n,j}\left(y_{\al,n,j}^{(m-1)}\right)}{h_{\al,n,j}'\left(y_{\al,n,j}^{(m-1)}\right)}\quad (m\ge1).
\end{equation}
Then $(y_{\al,n,j}^{(m)})_{m=0}^\infty$ is well defined and converges to
$z_{\al,n,j}$, and the convergence is at least linear:
\begin{equation}\label{eq:linear_newton_estimate}
y_{\al,n,j}^{(m)}-z_{\al,n,j}
\le \frac{\pi}{n} \left( \frac{\ka^2-1}{\ka n-1}\right)^m.
\end{equation}
Moreover, if $n\ge 2 N_{\al}$, then the convergence is quadratic, and
\begin{equation}
\label{eq:quadratic_newton_estimate}
y_{\al,n,j}^{(m)}
-z_{\al,n,j}
\le
\frac{\pi}{n}
\left(\frac{\pi \ka^2}{2n^2}\right)^{2^m-1}.
\end{equation}
\end{thm}

\begin{proof}
Formulas for $\etatrig'$ and $\etatrig''$ (\eqref{eq:eta_d1} and~\eqref{eq:eta_d2}) imply that $h_{\al,n,j}'>0$
and $h_{\al,n,j}''>0$ on $\clos(I_{n,j})$.
Moreover,
$d_{n,j-1}<z_{\al,n,j}<y_{\al,n,j}^{(0)}=d_{n,j}$.
So, the assumptions of Proposition~\ref{prop:convergence_f_convex_y0>c} are satisfied.
Here are rough estimates of the derivatives of $h_{\al,n,j}$ at the extremes of $I_{n,j}$:
\[
 n-\ka = h_{\al,n,j}'(0) \le h_{\al,n,j}'(d_{n,j-1}) \le h_{\al,n,j}'(d_{n,j}) \le h_{\al,n,j}'(\pi) = n-\frac{1}{\ka}.
\]
Therefore,
\[
1-\frac{h_{\al,n,j}'(d_{n,j-1})}{h_{\al,n,j}'(d_{n,j})} \le 1-\frac{h_{\al,n,j}'(0)}{h_{\al,n,j}'(\pi)} = \frac{\ka^2-1}{\ka n-1},
\]
and we obtain~\eqref{eq:linear_newton_estimate}.

Finally, if $n \geq 2N_{\al}$, then 
\[
\frac{\pi}{n} \cdot \frac{\displaystyle\max_{0 \le x \le \pi}|h_{\al,n,j}''(x)|}{2\displaystyle\min_{0 \le x\le \pi}|h_{\al,n,j}'(x)|} 
\le 
\frac{\pi \ \displaystyle\max_{0 \le x \le \pi} |\etatrig''(x)|}{2n  \displaystyle\left(n- \max_{0 \le x \le \pi}|\etatrig'(x)|\right)} 
\le 
\frac{\pi(\ka^2-1)}{4n(n-\ka)} 
\le 
\frac{\pi \ka^2}{2n^2} < 1,
\]
which implies the quadratic convergence
with upper estimate~\eqref{eq:quadratic_newton_estimate};
see, e.g.,~\cite[Section~2.2]{A1989}
or~\cite[Proposition~26]{GMS2022}.
\end{proof}

The initial condition $y_{\al,n,j}^{(0)}=d_{n,j}$
assures that we start from the correct side of the root.
Otherwise, the rule~\eqref{eq:Newton_sequence_eigval} can yield a point greater than 
$(j-1)\pi/n$ or even than $\pi$.

For $m$ large enough, $y_{\al,n,j}^{(m)}$ approaches to $z_{\al,n,j}$, and the convergence in Theorem~\ref{thm:Newton} becomes quadratic, according to the general theory of Newton's method.

\begin{remark}\label{rem:n<ka}
    For $- \frac{1}{2} < \al < 0$, we have that $\ka>3$.
    Let us explain the situation with the eigenvalues for $3\le n \le \ka$.
    \begin{itemize}
        \item For every even $j$ with $2\le j \le n$ (except for $n=\ka$ and $j=2$), $z_{\al,n,j}$ satisfies equation $h_{\al,n,j}(x) = 0$ and can be computed by the bisection method.
        \item Since $\|\etatrig'\|_\infty / n \ge 1$, we cannot guarantee the convergence of fixed point method for all $j$.
        \item $h_{\al,n,j}'(x)$ can vanish for some $j$ and $x$, and we cannot guarantee the convergence of Newton's method.
    \end{itemize}
\end{remark}

Now we pass to the asymptotic analysis of the eigenvalues $\la_{\al,n,j}$ with even $j$ such that $4\le j\le n$, as $n\to\infty$.

By Theorem~\ref{thm:main_equations}, for every $n\ge N_\al$, and every even $j$ with $4\le j\le n$,
\[
|z_{\al,n,j}-d_{n,j}|\le\frac{\pi}{n}.
\]

The proofs of the next Propositions~\ref{prop:theta_approximation_1} and~\ref{prop:weak_theta_asympt} are very similar to the proofs given in~\cite[Propositions 29 and 30]{GMS2022}. 

    \begin{prop}
\label{prop:theta_approximation_1}
Let $n\ge N_\al$ and $j$ be even with $4\le j\le n$. 
Then
\begin{equation}
\label{eq:theta_approximation_1}
\left|z_{\al,n,j}-\left(d_{n,j}+\frac{\eta_\al(d_{n,j})}{n}\right)\right|
\le \frac{\pi \ka}{n^2}.
\end{equation}

\end{prop}
    \begin{prop}
\label{prop:weak_theta_asympt}
There exists $C_3(\al)>0$ such that
for every $n\ge N_\al$ and every even $j$ with $4\le j\le n$,
\begin{equation}\label{eq:weak_theta_asympt}
z_{\al,n,j}
=d_{n,j}
+\frac{\eta_\al(d_{n,j})}{n}
+\frac{\eta_\al(d_{n,j})\eta_\al'(d_{n,j})}{n^2}
+r_{\al,n,j},
\end{equation}
where $|r_{\al,n,j}|\le\frac{C_3(\al)}{n^3}$.
\end{prop}

\begin{proof}[Proof of Theorem~\ref{thm:inner_eigvals_expansion}]
    Substituting~\eqref{eq:weak_theta_asympt} into $g$
    and using Taylor expansion of $g$ around $d_{n,j}$, we obtain the asymptotic expansion~\eqref{eq:laasympt_w} with error bound~\eqref{eq:weak_lambda_asympt_K}.
\end{proof}

\section{First eigenvalue}
\label{sec:out_eig_left}

In this section, we suppose that $\al<0$ and $n\ge N_\al$, and we analyze the behavior of $\la_{\al,n,1}$ as $n\to\infty$.
Recall that $\ka$, $\phihyp$, and $N_\al$ are defined respectively by~\eqref{eq:ka},~\eqref{eq:phi}, and~\eqref{eq:N_al}.


\begin{proof}[Proof of Theorem~\ref{thm:main_equations}]
Let $n\ge N_\al$. 
If $\la<0$, we use the change of variable $\la=g_-(x)$ with $x\in(0,\infty)$, and obtain~\eqref{eq:char_pol_gminus}.
Then, $D_{\al,n}(g_-(x)) =0$ takes the form $q_{\al,n}(2\cosh(x/2))=0$, i.e.,
\begin{equation}\label{eq:tangents_hyp_ls}
\tanh\frac{n x}{2} = \ka\tanh\frac{x}{2}.
\end{equation}
By Theorem~\ref{thm:localization_eigvals_left}, there is a unique solution in $(0,\infty)$
of~\eqref{eq:tangents_hyp_ls}, namely $\smin$. 
Dividing both sides of~\eqref{eq:tangents_hyp_ls} by $\ka$ and applying $\arctanh$, we rewrite this equation in the form~\eqref{eq:eq_first}. 
The second part of Theorem~\ref{thm:main_equations} is proved in the beginning of Section~\ref{sec:inner_eigvals_left}.
\end{proof}

The main advantage of equation~\eqref{eq:tangents_hyp_ls} is that $\phihyp$ is a ``very slow function'' for big values of $n$.
A straightforward computation yields
\begin{align}
    \label{eq:zeta_left_der_1}
    \phihyp^\prime(x) & = \frac{n\ka}{\left(\ka^2-\tanh^2\frac{n x}{2}\right)\cosh^2\frac{n x}{2}}=\frac{n\al(\al-1)}{\al^2+(1-2\al)\cosh^2\frac{n x}{2}}, 
    \\
    \label{eq:zeta_left_der_2}
    \phihyp^{\prime\prime}(x) & = -\frac{n^2\al(\al-1)(1-2\al) \cosh\frac{nx}{2} \sinh\frac{nx}{2}}{\left(\al^2+(1-2\al)\cosh^2\frac{n x}{2}\right)^2}.
\end{align}
Recall that $\Omal$ and $\omal$ are defined by~\eqref{eq:Om_al} and~\eqref{eq:om_al}.
Define
\[
\lbnd \eqdef \frac{2}{ n }\arccosh\sqrt{\frac{n \al(\al-1) - \al^2}{1-2\al}}.
\]

\begin{prop}\label{prop:properties_zeta}
Let $n \ge N_\al$. 
Then $\phihyp$ has the following properties.
\begin{enumerate}
    \item $\phihyp$  is strictly increasing and strictly concave.
    \item $\phihyp'(\lbnd) = 1$; moreover, $\phihyp'> 1$ on $[0,\lbnd)$ and $\phihyp' < 1$ on $(\lbnd,+\infty)$.
    \item $\lim_{x\to+\infty}\phihyp(x) = \omal$.
    \item  $\smin$ is the unique fixed point of $\phihyp$ on $(0,+\infty)$.
    \item $\phihyp(x) > x$ for every $x$ in $(0,\lbnd]$.
    \item $\lbnd <\smin$.
\end{enumerate}
\end{prop}

\begin{proof}
    Properties 1 and 2 follow from~\eqref{eq:phi},~\eqref{eq:zeta_left_der_1}, and~\eqref{eq:zeta_left_der_2}.
    The limit in 3 is easy to compute taking into account that $\tanh(\omal/2) = \ka^{-1}$.
    Property 4 follows from Theorem~\ref{thm:main_equations}.

    To prove 5, we apply the mean value theorem to $\phihyp$ on the segment $[0,x]$, taking into account property 2.
    
     Let us prove 6.     
     Due to property 5, we have that $\phihyp(x) > x$ for every $x$ in $(0,\lbnd]$.
     Hence, the fixed point of $\phihyp$ cannot belong to $(0,\lbnd]$.
     On the other hand, the function
     $x\mapsto x-\phihyp(x)$ is continuous and changes its sign on $[\lbnd,+\infty)$.
     Therefore, $\phihyp$ has a fixed point on $(\lbnd,+\infty)$.
\end{proof}

Figure~\ref{fig:phiAndIdentity} shows $\phihyp$ together with the identity function.

\begin{figure}[htb]
\centering
\includegraphics{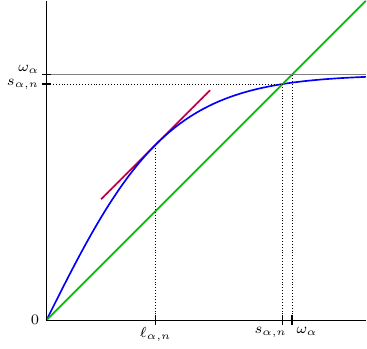}
\caption{Plot of $\phihyp$ (blue), tangent line to the graph of $\phihyp$ at $\lbnd$ (purple), and plot of $x\mapsto x$ (green), for $\al = -1/2$ and $n= 6$.
\label{fig:phiAndIdentity}
}
\end{figure}

\begin{thm}\label{thm:phi_fixed_point}
    Let $n\ge N_\al$.
    Then $\phihyp$ is a contraction on $[\phihyp(\lbnd),\omal]$.
\end{thm}
\begin{proof}
    As we have already mentioned in Proposition~\ref{prop:properties_zeta}, $\phihyp$ strictly increases, and $\phihyp'$ strictly decreases.
    Moreover, by property 5 from Proposition~\ref{prop:properties_zeta}, $\lbnd < \phihyp(\lbnd)$.
    Therefore, for every $x$ in $[\phihyp(\lbnd),\omal]$, 
    \[
    \phihyp(\lbnd) < \phihyp(\phihyp(\lbnd)) \le  \phihyp(x) < \omal,
    \]
    and
    \[     
        0<\phihyp'(x) \le \phihyp'(\phihyp(\lbnd)) <\phihyp'(\lbnd) =1.
    \]
    So, $\phihyp([\phihyp(\lbnd),\omal])\subseteq [\phihyp(\lbnd),\omal]$, and $\phihyp'(\phihyp(\lbnd))$ is a Lipschitz coefficient for $\phihyp$ on $[\phihyp(\lbnd),\omal]$.
\end{proof}

For every $n\ge N_\al$, we define $\fmin\colon [0,+\infty)\to\bR$ by
\begin{equation*}
\fmin(x) \eqdef x-\phihyp(x).
\end{equation*}
Figure~\ref{fig:xminusphi} shows $f_{\al,n}$.

\begin{figure}[htb]
\centering
\includegraphics{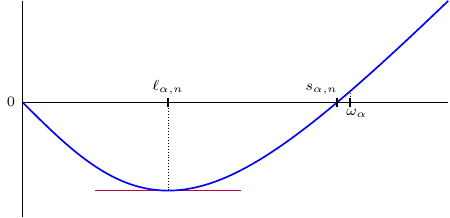}
\caption{Plot of $f_{\al,n}$ (blue) and tangent line to the graph of $f_{\al,n}$ at $\lbnd$ (purple), for $\al = -1/2$ and $n= 6$.
\label{fig:xminusphi}
}
\end{figure}

\begin{prop}\label{prop:f_properties}
Let $n\ge N_\al$. 
Then $\fmin$ has the following properties.
\begin{enumerate}
    \item $\fmin''>0$ on $[0,+\infty)$, and $\fmin$ is strictly convex.
    \item $\fmin'$ is strictly negative on $[0,\lbnd)$ and strictly positive on $(\lbnd,+\infty)$.
    \item $\lim_{x\to\infty} \fmin(x) = +\infty$.
    \item $\smin$ is the only root of $\fmin$ in $(0,+\infty)$.
    \item\label{item1} $\fmin$ is strictly negative on $(0,\smin)$ and strictly positive on $(\smin,+\infty)$.
\end{enumerate}
\end{prop}

\begin{proof}
Properties 1--4 follow from Proposition~\ref{prop:properties_zeta}.
To prove property 5, we also apply the intermediate value theorem.
\end{proof}

\begin{thm}[convergence of Newton's method applied to $f_{\al,n}$]\label{thm:linear_conv_first_eigval}
Let $n\ge N_\al$. 
Then the sequence defined $(y_{\al,n}^{(m)})_{m=0}^{\infty}$ by
\begin{equation*}
y_{\al,n}^{(0)} \eqdef \omal ,\qquad
y_{\al,n}^{(m)}
\eqdef y_{\al,n}^{(m-1)} - \frac{\fmin\left(y_{\al,n}^{(m-1)}\right)}{\fmin'\left(y_{\al,n}^{(m-1)}\right)}\quad (m\ge1),
\end{equation*}
takes values in $[\smin,\omal]$ and converges to $\smin$. 
\end{thm}

\begin{proof}
Theorem~\ref{thm:localization_eigvals_left}  and Proposition~\ref{prop:properties_zeta} yield $\lbnd<\smin<\omal$. 
By Proposition~\ref{prop:f_properties}, $\fmin'>0$ and $\fmin''>0$ on $[\smin,\omal]$.
The conclusion now follows from Proposition~\ref{prop:convergence_f_convex_y0>c}. 
\end{proof}

The convergence in Theorem~\ref{thm:linear_conv_first_eigval} is at least linear convergence, but it becomes quadratic after a finite number of steps.

\begin{prop}\label{prop:fixed_point_increasing}
The sequence $(\smin)_{n\ge N_\al}$ is strictly increasing: if  $n>m\ge N_\al$, then $\smin >s_{\al,m}$.
\end{prop}

\begin{proof}
It follows directly from~\eqref{eq:eta} that $\varphi_{\al,m}(x)$ strictly increases with respect to $m$.
Thereby, for $n>m\ge N_\al$ and for every $x>0$,
\[
\fmin(x) < f_{\al,m}(x).
\]
In particular,
\[
\fmin(s_{\al,m}) < f_{\al,m}(s_{\al,m}) = 0.
\]
Property~\ref{item1} from Proposition~\ref{prop:f_properties} implies that $\smin>s_{\al,m}$.
\end{proof}

\begin{cor}\label{cor:s>lN}
    For every $n\ge N_\al$, $\smin > \ell_{\al,N_\al}$.
\end{cor}

\begin{proof}
    Indeed, $\smin > s_{\al,N_\al} > \ell_{\al,N_\al}$ by property~6 from Proposition~\ref{prop:properties_zeta} and Propostion~\ref{prop:fixed_point_increasing}.
\end{proof}

\begin{prop}\label{prop:thtn_convergence_omal_exponentially_left}
If $n\ge N_\al$, then
\begin{equation}\label{eq:thtminus_exp_conv_omal}
     0 \le \omal - \smin \le C_4(\al) e^{-n \omal},
\end{equation}
where
\begin{equation*}
    C_4(\al) \eqdef  \frac{4\cosh^2\frac{\omal}{2}}{\ka}  \exp\left(\frac{4\cosh^2\frac{\omal}{2}}{e\, \ell_{\al,N_\al} \ka}\right) = \frac{4\ka}{(\ka^2-1)}\exp\left(\frac{4\ka}{e\, (\ka^2-1) \ell_{\al,N_\al} }\right) .
\end{equation*}
\end{prop}

\begin{proof}
By the mean value theorem applied to $\tanh(x/2)$ on $[\smin,\omal]$, there exists $\xi$ in $(\smin,\omal)$ such that
\[
\tanh\frac{\omal}{2}- \tanh\frac{\smin}{2} = \frac{1}{2\cosh^2\frac{\xi}{2}} (\omal-\smin).
\]
Notice that $\cosh(\xi/2) < \cosh(\omal/2)$. Hence, we can estimate $\omal-\smin$ from above:
\[
\omal-\smin \le 2\cosh^2\frac{\omal}{2}\left(\tanh\frac{\omal}{2}- \tanh\frac{\smin}{2}\right).
\]
Since $\smin$ satisfies~\eqref{eq:tangents_hyp_ls} and $\tanh(\omal/2) = \ka^{-1}$,
\begin{equation}
\label{eq:omega_minus_s_first_upper_bound}
\omal-\smin
\le \frac{2\cosh^2\frac{\omal}{2}}{\ka}\left(1-\tanh\frac{n\smin}{2}\right) \\
\le \frac{4\cosh^2\frac{\omal}{2}}{\ka} e^{-n\smin}.
\end{equation}
By Corollary~\ref{cor:s>lN}, $e^{-n\smin} < e^{-n \ell_{\al,N_\al}}$.
We also have the elementary inequality $xe^{-x}\le 1/e$ for every $x>0$.
By these inequalities and~\eqref{eq:omega_minus_s_first_upper_bound},
\[
n(\omal-\smin) \le \frac{4\cosh^2\frac{\omal}{2}}{\ka} ne^{-n\ell_{\al,N_\al}} = \frac{4\cosh^2\frac{\omal}{2}}{\ell_{\al,N_\al}\ka} n\ell_{\al,N_\al}e^{-n\ell_{\al,N_\al}} \le \frac{4\cosh^2\frac{\omal}{2}}{e\, \ell_{\al,N_\al}\ka}.
\]
Now we combine this inequality with~\eqref{eq:omega_minus_s_first_upper_bound}:
\[
\omal-\smin
\le \frac{4\cosh^2\frac{\omal}{2}}{\ka} e^{-n(\omal-\smin)}
e^{-n\omal}
\le
\frac{4\cosh^2\frac{\omal}{2}}{\ka}  \exp\left(\frac{4\cosh^2\frac{\omal}{2}}{e\, \ell_{\al,N_\al} \ka}\right)
e^{-n\omal}.\qedhere
\]
\end{proof}

Define
 \begin{equation}\label{eq:gammas}
    \ga_{1,\al} \eqdef  \frac{4\ka}{\ka^2-1}, \quad \ga_{2,\al} \eqdef  \frac{4\ka(\ka^2+1)}{(\ka^2-1)^2}.
\end{equation}

\begin{lem}[asymptotic expansion of $\varphi_{\al,1}$]\label{lem:phiExpansion}
As $t$ tends to infinity,
\begin{equation}\label{eq:phi1_expansion}
    \varphi_{\al,1}(t) = \omal - \ga_{1,\al} e^{-t} + \ga_{2,\al} e^{-2t} + O(e^{-3t}).
\end{equation}    
\end{lem}

\begin{proof}
Since $\tanh(t/2) = (1-e^{-t})/(1+e^{-t})$, 
\[
    \varphi_{\al,1}(t) = \psi(e^{-t}), \quad \text{where} \quad \psi(u) \eqdef \arctanh\left( \ka^{-1} \frac{1-u}{1+u} \right).
\]
We start with the Taylor--Maclaurin expansion of the rational function $u\mapsto (1-u)/(1+u)$ around $0$:
\[
\frac{1-u}{1+u}
=1-\frac{2u}{1+u}
=1-2u+2u^2+O(u^3).
\]
Then, we apply the Taylor expansion of $\arctanh$ around $\ka^{-1}$:
\[
\arctanh(\ka^{-1}+y)
=\arctanh(\ka^{-1})
+\frac{y}{1-\ka^{-2}}
+\frac{\ka^{-1}\,y^2}{(1-\ka^{-2})^2}
+O(y^3).
\]
In the last expansion we substitute $y=2\ka^{-1}(-u+u^2+O(u^3))$ and use the relation $O(y)=O(u)$:
\begin{align*}
\psi(u) &=
2\arctanh\left(\ka^{-1}
+2\ka^{-1}
(-u
+u^2
+O(u^3))\right) 
\\
&=
2\arctanh(\ka^{-1})
+ \frac{4\ka^{-1}}{1-\ka^{-2}}
\left(-u + u^2 + O(u^3)\right)
\\
&\qquad
+\frac{8\ka^{-3}}{(1-\ka^{-2})^2}
\left(-u + u^2 + O(u^3)\right)^2
+ O(u^3).
\end{align*}
Simplifying and taking into account that $\tanh(\omal/2) = \ka^{-1}$, we obtain the Taylor--Maclaurin expansion of $\psi$ around $0$:
\[
\psi(u) = \omal - \ga_{1,\al} u + \ga_{2,\al} u^2 + O(u^3).
\]
Finally, we put $u = e^{-t}$ and obtain~\eqref{eq:phi1_expansion}.
\end{proof}

\begin{thm}[asymptotic expansion of $\smin$]\label{thm:omin_expansion}
As $n$ tends to infinity,
\begin{equation}\label{eq:smin_expansion}
 \smin = \omal -\ga_{1,
 \al}e^{-n\omal} - \ga_{1,
 \al}^2 n e^{-2n\omal} + \ga_{2,
 \al} e^{-2n\omal} + O(n^2e^{-3n\omal}).
\end{equation}
\end{thm}

\begin{proof}
    By  formula~\eqref{eq:thtminus_exp_conv_omal} from Proposition~\ref{prop:thtn_convergence_omal_exponentially_left}, we have an asymptotic expansion of $\smin$ with one exact term:
    \begin{equation}\label{eq:smin_expansion0}
    \smin = \omal + O(e^{-n\omal}).
    \end{equation}
    Therefore,    
    \begin{equation}\label{eq:expns_1}
    e^{-n\smin} = e^{-n\omal+O(ne^{-n\omal})} = e^{-n\omal} (1+O(ne^{-n\omal})) =e^{-n\omal}+O(ne^{-2n\omal}).
    \end{equation}
    This also implies a rough upper bound for $e^{-n\smin}$:
    \begin{equation}\label{eq:expns_0}
    e^{-n\smin} = O(e^{-n\omal}).
    \end{equation}
    The main idea of the following proof is to combine~\eqref{eq:smin_expansion0} with~\eqref{eq:eq_first} and Lemma~\ref{lem:phiExpansion}.
    We apply the asymptotic expansion~\eqref{eq:phi1_expansion} with two exact terms and with $n\smin$ instead of $t$:
    \[
    \smin = \phihyp(\smin) = \varphi_{\al,1}(n\smin)  = \omal - \ga_{1,\al} e^{-n\smin} + O(e^{-2n\smin}).
    \]
    We simplify this expression using~\eqref{eq:expns_1} and~\eqref{eq:expns_0}:
    \begin{align*}
        \smin & = \omal - \ga_{1,\al} e^{-n\omal} + O(ne^{-2n\omal}) + O(e^{-2n\omal})
        \\
        & = \omal - \ga_{1,\al} e^{-n\omal} + O(ne^{-2n\omal}).
    \end{align*}
    Now, we use this expansion to improve~\eqref{eq:expns_1}:
    \begin{align*}
    e^{-n\smin} & = e^{-n\omal} e^{\ga_{1,\al} n e^{-n\omal} + O(n^2e^{-2n\omal})} \\
    & = e^{-n\omal}\left(1 + \ga_{1,\al}ne^{-2n\omal} + O(n^2e^{-2n\omal}) \right) \\
    & = e^{-n\omal} + \ga_{1,\al}ne^{-2n\omal} + O(n^2e^{-3n\omal}).
    \end{align*}
    Next, we combine this expansion with~\eqref{eq:phi1_expansion}:
    \begin{align*}
        \smin & = \phihyp(\smin) = \varphi_{\al,1}(n\smin)  = \omal - \ga_{1,\al} e^{-n\smin} + \ga_{2,\al} e^{-2n\smin} + O(e^{-3n\smin}) \\
        & = \omal -\ga_{1,\al}\left(e^{-n\omal} + \ga_{1,\al}ne^{-2n\omal} + O(n^2e^{-3n\omal})\right) \\
        & \pheq \quad + \ga_{2,\al} \left(e^{-n\omal} + \ga_{1,\al}ne^{-2n\omal} + O(n^2e^{-3n\omal})\right)^2 + O(e^{-3n\omal}).
    \end{align*}
    Simplifying this expression we get~\eqref{eq:smin_expansion}.
\end{proof}

\begin{remark}
    For $n$ large enough, Theorem~\ref{thm:omin_expansion} provides a more precise localization of $\smin$ than in Theorem~\ref{thm:phi_fixed_point} and Proposition~\ref{prop:thtn_convergence_omal_exponentially_left}. Namely, there exists $M_\al$ such that for $n\ge M_\al$,
    \[
      \omal-\ga_{1,\al}e^{-n\omal} - \ga_{1,
 \al}^2 n e^{-2n\omal} < \smin < \omal.
    \]
\end{remark}

\begin{proof}[Proof of Theorem~\ref{thm:left_strong_asympt}]
We expand $g_-$ by Taylor formula around $\omal$:
\[
g_-(\omal + x) =  g_-(\omal) + g_-'(\omal)x + \frac{g_-''(\omal)}{2}x^2 + O(x^3).
\]
Then we substitute the expansion~\eqref{eq:smin_expansion} of $\smin$: 
\begin{align*}
    \la_{\al,n,1} & = g_-(\smin) 
    \\
    & = g_-\left(\omal -\ga_{1,
 \al}e^{-n\omal} - \ga_{1,
 \al}^2 n e^{-2n\omal} + \ga_{2,
 \al} e^{-2n\omal} + O(n^2e^{-3n\omal})\right) \\
     & = g_-(\omal) + g_-'(\omal)  \left(-\ga_{1,
 \al}e^{-n\omal} - \ga_{1,
 \al}^2 n e^{-2n\omal} + \ga_{2,
 \al} e^{-2n\omal} + O(n^2e^{-3n\omal})\right) \\
    &\pheq \quad + \frac{g''(\omal)}{2}\left(-\ga_{1,
 \al}e^{-n\omal} - \ga_{1,
 \al}^2 n e^{-2n\omal} + \ga_{2,
 \al} e^{-2n\omal} + O(n^2e^{-3n\omal})\right)^2 \\
    & \pheq \quad +O(e^{-3n\omal}) \\
    & =  g_-(\omal)  -\ga_{1,\al} g_-'(\omal) e^{-n\omal} -\ga_{1,\al}^2 g_-'(\omal)ne^{-2n\omal}\\&\pheq \quad +  \left(\ga_{\al,2} g_-'(\omal) + \frac{\ga_{1,\al}^2g_-''(\omal)}{2}\right)e^{-2n\omal} + O\left(n^2e^{-3n\omal}\right).
\end{align*}
Recall that $g_-(\omal) = \Omal$.
Hence we obtain~\eqref{eq:left_laasympt} and~\eqref{eq:asmpt_min_three_terms}, with the following coefficients:
\[
\be_{\al,1} = -g_-'(\omal)\ga_{1,\al}, \quad \be_{\al,2} = -g_-'(\omal)
\ga_{1,\al}^2, \quad \be_{\al,3} = -g_-'(\omal)\ga_{\al,2}
-\frac{1}{2} g_-''(\omal)\ga_{1,\al}^2.
\]
Calculate the derivatives of $g_-$ at $\omal$:
\begin{align*}
 g_-'(\omal)
& = -2\sinh(\omal) = \frac{4\al(1-\al)}{1-2\al}
 =-\frac{4\ka}{\ka^2-1}, 
 \\
g_-''(\omal) & = -2\cosh(\omal) = -\frac{2(2\al^2-2\al+1)}{1-2\al}
= -\frac{2(\ka^2+1)}{\ka^2-1}.
\end{align*}
Combining with formulas~\eqref{eq:gammas}, we write $\be_{\al,1}$, $\be_{\al,2}$, and $\be_{\al,3}$ as~\eqref{eq:bs} or~\eqref{eq:bs2}.
\end{proof}

\section{Eigenvectors}\label{sec:eigvec_left}

We recall that $\la_{\al,n,j} = \la_{\Re(\al),n,j}$.
Nevertheless, it turns out that if $\Im(\al) \ne 0$, then the eigenvectors associated to $L_{\al,n}$ have complex components.
So, in this section we suppose that $\al$ belongs to $\bC$ and $\Re(\al)<0$.
To simplify subindices, we put 
\[
\ka \eqdef \varkappa_{\Re(\al)}, \quad 
N_\al \eqdef N_{\Re(\al)}, \quad 
\omal \eqdef \om_{\Re(\al)}, \quad 
\Omal \eqdef \Omega_{\Re(\al)},
\]
\[
\eta_\al \eqdef \eta_{\Re(\al)}, \quad
z_{\al,n,j}\eqdef z_{\Re(\al),n,j}, \quad
s_{\al,n} \eqdef s_{\Re(\al),n}.
\]

\begin{proof}[Proof of Theorem~\ref{thm:eigvec_ls}]
Formulas~\eqref{eq:eivec_ls},~\eqref{eq:eivec_w} are consequences of~\cite[Proposition 8]{GMS2022}.
\end{proof}

Recall that $\nu_\al$ is defined by~\eqref{eq:nu_al}.
For every $x\in[0,\pi]$, we define 
\begin{equation*}
   \begin{aligned}
      \xi_\al(x) & \eqdef
      \frac{|1-\al|^2}{2} g(x) \cos(\eta_\al(x)) + \frac{|\al|^2}{2}g(\eta_\al(x))\cos(x)\\
      &\pheq+\frac{\Re(\al)-|\al|^2}{2}\left(g(x) + g(x+\eta_\al(x)) -g(\eta_\al(x))\right) - 2|\al|^2\cos(x).
    \end{aligned}
\end{equation*}

\begin{prop}[exact formulas for the inner eigenvectors]
\label{prop:exact_norm_eigenvector}
Let $n\ge 3$ and $3\le j\le n$. 
If $j$ is odd, then $\|v_{\al,n,j}\|_2$ is given by~\eqref{eq:norm_eigvec_odd_left}.
If $j$ is even, then
\begin{equation}
\label{eq:exact_norm_eigenvector_even}
     \|v_{\al,n,j}\|_2^2 =  n\nu_\al(z_{\al,n,j}) + \frac{\sin(\eta_\al(z_{\al,n,j}))}{\sin(z_{\al,n,j})}\xi_\al(z_{\al,n,j}).
\end{equation}
\end{prop}

\begin{proof}
    These formulas are similar to~\cite[(66), (69)]{GMS2022} and are proved in the same manner.
\end{proof}

In this section, we use several identities for hyperbolic functions:
\begin{align} 
\label{eq:difference_sinh}
\sinh(x) \pm \sinh(y) & = 2\sinh\frac{x\pm y}{2} \cosh\frac{x \mp y}{2},
\\
\label{eq:product_cosh}
2\cosh(x)\cosh(y) &= \cosh(x-y)+\cosh(x+y),
\\
\label{eq:square_cosh}
2\cosh^2(x) & = 1+\cosh(2x),
\\
\label{eq:cosh_sum}
\sum_{k=1}^n \cosh(2 k x+y) & = \frac{\sinh(n x) \cosh((n+1)x+y)}{\sinh(x)}.
\end{align}
Define 
\begin{align}
\label{eq:u1}
u_{1,\al,n} &\eqdef -\frac{\la_{\al,n,1}}{2} \left(n +  \frac{ \sinh(2n\smin) }{2\sinh(\smin)} \right), 
\\
\label{eq:u2}
u_{2,\al,n} &\eqdef   2|\al|^2\sinh^2\frac{(n-1)\smin}{2} \left(n + \frac{\sinh(n\smin)}{\sinh(\smin)}\right), 
\\
\label{eq:u3}
u_{3,\al,n} & \eqdef 4\Re(\al) \sinh\frac{(n-1)\smin}{2} \cosh\frac{n\smin}{2} \sinh\frac{\smin}{2}
\left( n + \frac{\sinh(n\smin)}{\sinh(\smin)}\right).
\end{align}

\begin{prop}[exact formula for the norm of the first eigenvector]\label{prop:norm_v1_expansion}
    Let $n\ge N_{\al}$. 
    Then
    \begin{equation}
    \label{eq:exact_norm_v1_us}
    \|v_{\al,n,1}\|_2^2   = u_{1,\al,n} + u_{2,\al,n} +u_{3,\al,n}.
    \end{equation}
\end{prop}

\begin{proof}
We transform~\eqref{eq:eivec_ls} using~\eqref{eq:difference_sinh}:
\[ 
v_{\al,n,1,k} = 2 \sinh\frac{\smin}{2} \cosh\frac{(2k-1)\smin}{2}  + 2\overline{\al} \sinh\frac{(n-1)\smin}{2} \cosh\frac{(n+1-2k)\smin}{2}.
\]
Then, 
\begin{align*}
    |v_{\al,n,1,k}|^2 & = 4\sinh^2\frac{\smin}{2}\cosh^2\frac{(2k-1)\smin}{2}\\
    &\pheq + 4|\al|^2\sinh^2\frac{(n-1)\smin}{2}\cosh^2\frac{(n+1-2k)\smin}{2} \\
    & \pheq + 8\Re(\al) \sinh\frac{\smin}{2}  \sinh\frac{(n-1)\smin}{2} \cosh\frac{(2k-1)\smin}{2}\cosh\frac{(n+1-2k)\smin}{2}.
\end{align*}
Applying~\eqref{eq:product_cosh} and~\eqref{eq:square_cosh}, we simplify some products or squares containing $k\smin$:
\begin{align*}
    |v_{\al,n,1,k}|^2 & = -\frac{\la_{\al,n,1}}{2} \bigl(1+\cosh(2k\smin -\smin ) \bigr)
    \\ & \pheq+ 2|\al|^2 \sinh^2\frac{(n-1)\smin}{2} \bigl(1+\cosh(2k\smin -(n+1)\smin)\bigr) \\
    &\pheq + 4\Re(\al) \sinh\frac{(n-1)\smin}{2} \sinh\frac{\smin}{2} \times \\
    & \pheq \qquad \times \left(\cosh\frac{n\smin}{2} + \cosh\left(2k\smin - \frac{(n+2)}{2}\smin\right)\right).
\end{align*}
Finally, we sum over $k$,  use~\eqref{eq:cosh_sum}, and obtain~\eqref{eq:exact_norm_v1_us}.
\end{proof}

\begin{lem}\label{lem:expansions_ujs}
    As $n$ tends to infinity, expressions~\eqref{eq:u1},~\eqref{eq:u2}, and~\eqref{eq:u3} have the following asymptotic behavior:
    \begin{align}
    \label{eq:u1_expansion}
        u_{1,\al,n} &= -\frac{\Omal}{8\sinh(\omal)} e^{2n\omal} + O(ne^{n\omal}), 
        \\
        \label{eq:u2_expansion}
        u_{2,\al,n} &= \frac{|\al|^2 e^{-\omal}}{4\sinh(\omal)}e^{2n\omal} + O(ne^{n\omal}), \\
        \label{eq:u3_expansion}
        u_{3,\al,n} &= \frac{\Omal}{8\sinh(\omal)} e^{2n\omal} + O(ne^{n\omal}).
    \end{align}
\end{lem}

\begin{proof}
We are going to prove~\eqref{eq:u1_expansion}; the proofs of~\eqref{eq:u2_expansion} and~\eqref{eq:u3_expansion} are similar.
As $n\to\infty$, by Theorem~\ref{thm:left_strong_asympt},
\begin{equation}\label{eq:la1_expansion_0}
\la_{\al,n,1} = \Omal + O(e^{-n\omal}).
\end{equation}
By~\eqref{eq:smin_expansion0}, 
\[
e^{2n\smin} = \exp(2n\omal + O(ne^{-n\omal})) = e^{2n\omal} (1 + O(ne^{-n\omal})) = e^{2n\omal} +O(ne^{n\omal}),
\]
\[
\frac{1}{\sinh(\smin)} = \frac{1}{\sinh(\omal)} + O(e^{-n\omal}).
\]
Therefore
\begin{equation}\label{eq:sinh(ns)cosh(ns)}
  n+ \frac{ \sinh(2n\smin) }{2\sinh(\smin)} = \frac{1}{4\sinh(\omal)} e^{2n\omal} + O(ne^{n\omal}).
\end{equation}
We conclude the proof by substituting~\eqref{eq:la1_expansion_0} and~\eqref{eq:sinh(ns)cosh(ns)} into~\eqref{eq:u1}.
\end{proof}

\begin{proof}[Proof of Theorem~\ref{thm:norms_eigvec_ls}]
Formula~\eqref{eq:norm_eigvec_even_left} follows from the exact formula~\eqref{eq:exact_norm_v1_us}, using the approximation~\eqref{eq:theta_approximation_1}, similarly to~\cite[(19)]{GMS2022}.
    To prove~\eqref{eq:norm_eigvec_left_1}, we apply Proposition~\ref{prop:norm_v1_expansion} and Lemma~\ref{lem:expansions_ujs}.
    The principal terms in the expansions~\eqref{eq:u1_expansion} and~\eqref{eq:u3_expansion} mutually annihilate, and
    \[
    \|v_{\al,n,1}\|_2^2  = \frac{|\al|^2 e^{-\omal}}{4\sinh(\omal)} e^{2n\omal} + O(ne^{n\omal}).
    \]
    By substituting $\omal = \log(1-2\Re(\al))$, we transform the coefficient into $\mu_\al^2$, where $\mu_\al$ is defined by~\eqref{eq:mu}.    
    Finally, we take the square root and obtain~\eqref{eq:norm_eigvec_left_1}.
\end{proof}

For $n\ge N_\al$ and $j=1$, we define
\begin{equation}\label{eq:whyp(x)}
    w_{\al,n,1}(x)\eqdef \sinh(x\smin)  - (1-\overline{\al}) \sinh((x-1)\smin) + \overline{\al} \sinh((n-x)\smin).
\end{equation}
For all other values of $n$ and $j$,
\begin{equation}\label{eq:w(x)}
w_{\al,n,j}(x)\eqdef \sin(xz_{\al,n,j})  - (1-\overline{\al}) \sin((x-1)z_{\al,n,j}) + \overline{\al} \sin((n-x)z_{\al,n,j}).
\end{equation}
We notice that~\eqref{eq:w(x)}
can be written as
\[
w_{\al,n,j}(x) = A_{\al,n,j} \sin(z_{\al,n,j}x + B_{\al,n,j}),
\]
for some coefficients $A_{\al,n,j}$ and $B_{\al,n,j}$.
Therefore, for $\al<0$
(here we suppose that $\al$ is real),
$w_{\al,n,j}$ changes its sign approximately $j-1$ times on $[0,\pi]$.

Figure~\ref{fig:eigvec_weak} shows the behavior of the formula for the eigenvectors~\eqref{eq:eivec_w}.
Notice that $\sqrt{\nu_\al(d_{n,j})}$ is an approximation of the quadratic mean of $|v_{\al,n,j,k}|$, $1\le k\le n$.

\begin{figure}[htb]
\centering
\includegraphics{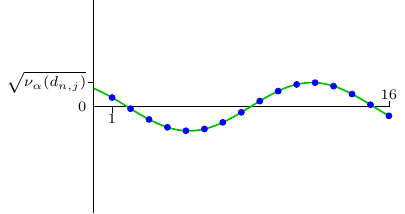}
\quad
\includegraphics{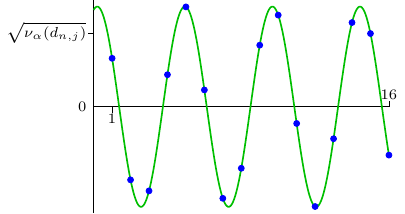}
\caption{Plots of $w_{\al,n,j}$ (green) and points $(k,v_{\al,n,j,k})$ (blue), for $\al = -1/2$, $n=16$, and $j=4, 8$.
\label{fig:eigvec_weak}
}
\end{figure}

Figure~\ref{fig:eigvec_hyp} shows $w_{\al,n,1}$ defined by~\eqref{eq:whyp(x)} and the components of $v_{\al,n,1}$.
We observe that the extreme components of the vector $v_{\al,n,1}$ are much bigger (in the absolute value) than their central components.

\begin{figure}[htb]
\centering
\includegraphics{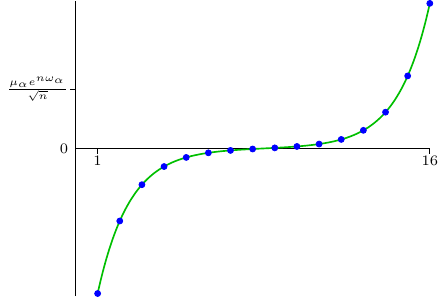}
\caption{Plot of $w_{\al,n,1}$ (green) and points $(k,v_{\al,n,1,k})$ (blue), for $\al = -1/2$ and $n=16$.
\label{fig:eigvec_hyp}
}
\end{figure}


\section{Numerical experiments}~\label{sec:num_exp}

With the help of Sagemath~\cite{Sage2023}, we have verified numerically (for many values of parameters) the representations~\eqref{eq:char_pol_fact},~\eqref{eq:charpol_factorization_trig},~\eqref{eq:char_pol_gminus}, for the characteristic polynomial, exact formulas~\eqref{eq:norm_eigvec_odd_left},~\eqref{eq:exact_norm_eigenvector_even},~\eqref{eq:exact_norm_v1_us} for the norms of the eigenvectors, and many other exact formulas appearing in this paper.

We introduce the following notation for different approximations of the eigenvalues and eigenvectors.
All computations are performed with $3322$ binary digits
($\approx 1000$ decimal digits).
\begin{itemize}
\item $\la_{\al,n,j}^{\text{gen}}$ are the eigenvalues computed in Sagemath by its general algorithms.
The multi-precision arithmetic versions of these algorithms are recently added to Sagemath, and they are not very accurate.

\item $z_{\al,n,j}^{\text{N}}$ is the numerical solution of the equation $h_{\al,n,j}(x)=0$ computed by Newton's method, see Theorem~\ref{thm:Newton}. 

\item Similarly, $\smin^{\text{N}}$ is the solution of $f_{\al,n}(x) = 0$ computed by Newton's method,
see Theorem~\ref{thm:linear_conv_first_eigval}.

\item $\la_{\al,n,j}^{\text{N}}$ is computed as $g(z_{\al,n,j}^{\text{N}})$ or $g(d_{n,j})$ or $g_-(\smin^{\text{N}})$, depending on the case.

\item $\la_{\al,n,j}^{\text{bisec}}$
is similar to $\la_{\al,n,j}^{\text{N}}$,
but now we solve the corresponding equations by the bisection method.

\item $\la_{\al,n,j}^{\text{fp}}$
is similar to $\la_{\al,n,j}^{\text{N}}$,
but now we solve the corresponding equations by the fixed point method.

\item Using $z_{\al,n,j}^{\text{N}}$ 
we compute $v_{\al,n,j}$ by~\eqref{eq:eivec_w} and normalize it.

\item Using $\smin^{\text{N}}$ we compute $v_{\al,n,1}$ by~\eqref{eq:eivec_ls} and normalize it.

\item $\la_{\al,n,j}^{\text{asympt}}$
is the approximation given 
by~\eqref{eq:laasympt_w}
and~\eqref{eq:left_laasympt}.
\end{itemize}

We have constructed a large series of examples including all rational values $\al$ in $[-3,0)$ with denominators $\le 3$
and all $n$ with $N_\al\le n\le 256$.
In all these examples, we have obtained
\[
\max_{1\le j\le n}
\|L_{\al,n}v_{\al,n,j}-\la_{\al,n,j}^{\text{N}}v_{\al,n,j}\|_2
<10^{-996},\qquad
\max_{1\le j\le n}|\la_{\al,n,j}^{\text{gen}}-\la_{\al,n,j}^{\text{N}}|
< 10^{-792},
\] 
\[
\max_{1\le j\le n}
|\la_{\al,n,j}^{\text{N}}-\la_{\al,n,j}^{\text{bisec}}|
<10^{-998},\qquad 
\max_{1\le j\le n}
|\la_{\al,n,j}^{\text{fp}}-\la_{\al,n,j}^{\text{N}}|
<10^{-998} .
\]
For testing the asymptotic formulas, we have computed the errors
\[
R_{\al,n,j}^{\text{asympt}}\eqdef \la_{\al,n,j}^{\text{asympt}}-\la_{\al,n,j}^{\mathrm{N}}
\]
and their maximums
$\|R_{\al,n}^{\text{asympt}}\|_\infty=\max_{1\le j\le n}|R_{\al,n,j}^{\text{asympt}}|$.
Table~\ref{table:errors_asympt_left} shows that these errors indeed can be bounded by $O_\al(1/n^3)$.

\begin{table}[htb]
\caption{Values of $\|R_{\al,n}^{\text{asympt}}\|_\infty$
and $n^3 \|R_{\al,n}^{\text{asympt}}\|_\infty$
for some $\al$ and $n$.
\label{table:errors_asympt_left}}
\[
\begin{array}{|c|c|c|}
\hline
\multicolumn{3}{ |c| }{\bigstrut\al=-1/3}
\\\hline
\bigstrut n & \|R_{\al,n}^{\text{asympt}}\|_\infty &%
\hstrut{}n^3 \|R_{\al,n}^{\text{asympt}}\|_\infty\hstrut{} 
\\\hline
\medstrut 128 & 2.84 \times10^{-5} & 59.6
\\
\medstrut 256 & 4.15 \times10^{-6} & 69.6
\\
\medstrut 512 & 5.54 \times10^{-7} & 74.4
\\
\medstrut 1024 & 7.14 \times10^{-8} & 76.7
\\
\medstrut 2048 & 9.06\times10^{-9} & 77.8
\\
\medstrut\hstrut{}4096\hstrut{}&%
\hstrut{}1.14\times10^{-9}\hstrut{}&%
\hstrut{}78.4\hstrut{}
\\
\medstrut\hstrut{}8192\hstrut{}&%
\hstrut{}1.43\times10^{-10}\hstrut{}&%
\hstrut{}78.7\hstrut{}
\\
\hline
\end{array}
\qquad
\begin{array}{|c|c|c|}
\hline
\multicolumn{3}{ |c| }{\bigstrut\al=-5/4}
\\\hline
\bigstrut n & \|R_{\al,n}^{\text{asympt}}\|_\infty &%
\hstrut{}n^3 \|R_{\al,n}^{\text{asympt}}\|_\infty\hstrut{}%
\\\hline
\medstrut 128 & 1.54 \times10^{-5} & 32.2
\\
\medstrut 256 & 2.02\times10^{-6} & 33.9\\
\medstrut 512 & 2.59\times10^{-7} & 34.8\\
\medstrut 1024 & 3.27\times10^{-8} & 35.1 \\
\medstrut 2048 & 4.11\times10^{-9} & 35.3 \\
\medstrut\hstrut{}4096\hstrut{}&%
\hstrut{}5.16\times10^{-10}\hstrut{}&%
\hstrut{}35.4 \hstrut{}\\
\medstrut\hstrut{}8192\hstrut{}&%
\hstrut{}6.45\times10^{-11}\hstrut{}&%
\hstrut{}35.5 \hstrut{}\\
\hline
\end{array}
\]
\end{table}

We have done similar tests for many other values of $\al$ and $n$. 
Numerical experiments show that $n^3\|R^{\text{asympt}}_{\al,n}\|_\infty$ are bounded by some numbers depending on $\al$.

Since $|R_{\al,n,j}^{\text{asympt}}|$ is much smaller for the outlier eigenvalue ($j=1$), we show in Table~\ref{table:errors_asympt_left_first} some numerical experiments for this case only.
We observe that $|R_{\al,n,1}^{\text{asympt}}|$ is bigger for bigger values of $|\Re(\al)|$.

\begin{table}[htb]
\caption{Values of $|R_{\al,n,1}^{\text{asympt}}|$
and $n^{-2}e^{3n\omal} |R_{\al,n,1}^{\text{asympt}}|$
for some $\al$ and $n$.
\label{table:errors_asympt_left_first}}
\[
\begin{array}{|c|c|c|}
\hline
\multicolumn{3}{ |c| }{\bigstrut\al=-1/3}
\\\hline
\hugestrut n & |R_{\al,n,1}^{\text{asympt}}| &%
\hstrut{} \dfrac{|R_{\al,n,1}^{\text{asympt}}|}{n^2 e^{-3n\omal}} \hstrut{} 
\\\hline
\medstrut 8 & 4.48\times10^{-4} & 1.48
\\
\medstrut 16 & 8.87\times10^{-9} & 1.54
\\
\medstrut 32 & 8.94\times10^{-19} & 1.73
\\
\medstrut 64 & 1.91\times10^{-39} & 1.84
\\
\medstrut 128 & 2.00\times10^{-81} & 1.89
\\
\medstrut 256 & 5.23\times10^{-166} & 1.92
\\
\medstrut 512 & 8.79\times10^{-336} & 1.93
\\
\hline
\end{array}
\qquad
\begin{array}{|c|c|c|}
\hline
\multicolumn{3}{ |c| }{\bigstrut\al=-5/4}
\\\hline
\hugestrut n & |R_{\al,n,1}^{\text{asympt}}| &%
\hstrut{} \dfrac{|R_{\al,n,1}^{\text{asympt}}|}{n^2 e^{-3n\omal}} \hstrut{}%
\\\hline
\medstrut 8 & 6.91\times10^{-10} & 1.23 \times 10^2
\\
\medstrut 16 & 2.77\times10^{-22} & 1.41 \times 10^2
\\
\medstrut 32 & 9.06\times10^{-48} & 1.50 \times 10^2
\\
\medstrut 64 & 2.20\times10^{-99} & 1.55 \times 10^2
\\
\medstrut 128 & 3.09\times10^{-203} & 1.58 \times 10^2
\\
\medstrut 256 & 1.49\times10^{-411} & 1.59 \times 10^2
\\
\medstrut 512 & 8.57\times10^{-829} & 1.60 \times 10^2
\\
\hline
\end{array}
\]
\end{table}

\medskip\noindent
Sergei M. Grudsky,\\ CINVESTAV del IPN,
Departamento de Matem\'aticas,
Ciudad de M\'exico,
Mexico.\\
\href{mailto:grudsky@math.cinvestav.mx}{grudsky@math.cinvestav.mx},\\
\url{ https://orcid.org/0000-0002-3748-5449},\\
\url{https://publons.com/researcher/2095797/sergei-m-grudsky}.

\medskip\noindent
Egor A. Maximenko,\\ Instituto Polit\'ecnico Nacional,
Escuela Superior de F\'isica y Matem\'aticas,
Ciudad de M\'exico,
Mexico.\\
\href{mailto:emaximenko@ipn.mx}{emaximenko@ipn.mx},\\
\url{https://orcid.org/0000-0002-1497-4338}.

\medskip\noindent
Alejandro Soto-Gonz\'{a}lez,\\  CINVESTAV  del IPN,
Departamento de Matem\'aticas,
Ciudad de M\'exico,
Mexico.\\
\href{mailto:asoto@math.cinvestav.mx}{asoto@math.cinvestav.mx},\\
\url{https://orcid.org/0000-0003-2419-4754}.

\end{document}